\documentclass[11pt]{article}
\usepackage{amsmath,amssymb,amsthm,nicefrac,mathrsfs}
\usepackage[colorlinks]{hyperref}
\usepackage[a4paper, margin=4cm, footskip=1cm]{geometry}
\newtheorem{theorem}{Theorem}[section]
\newtheorem{exmp}[theorem]{Example}% same for example numbers
\newtheorem{lemma}[theorem]{Lemma}

\newtheorem{corollary}[theorem]{Corollary}
\newtheorem{assumption}[theorem]{Assumption}

\numberwithin{equation}{section}

\newcommand{\be}{\begin{equation}}
\newcommand{\ee}{\end{equation}}

\newcommand{\bes}{\begin{equation*}}
\newcommand{\ees}{\end{equation*}}

\def\E{\bE}

\def\bE{\mathbb{E}}

\newcommand{\R}{\mathbf{R}}

\renewcommand{\d}{{\rm d}}

\renewcommand{\geq}{\geqslant}
\renewcommand{\leq}{\leqslant}

\def\m1{\mathbf{1}}

\allowdisplaybreaks[1]

\allowdisplaybreaks[1]

\author{Mohammud Foondun\\University of Strathclyde \and Ngartelbaye Guerngar\\Auburn University\\
\and
Erkan Nane\\ Auburn University}
\title{Some properties of non-linear fractional stochastic heat equations on bounded domains
\date{}
}

\begin{document}

\maketitle

\begin{abstract}
Consider the following stochastic partial differential equation,
\begin{equation*}
\partial_t u_t(x)=
 \mathcal{L}u_t(x)+ \xi\sigma (u_t(x)) \dot F(t,x),
\end{equation*}
where $\xi$ is a positive parameter and $\sigma$ is a globally Lipschitz continuous function. The stochastic forcing term $\dot F(t,x)$ is white in time but possibly colored in space.  The operator $\mathcal{L}$ is a non-local operator. We study the behaviour of the solution with respect to the parameter $\xi$, extending the results in \cite{FoonNual} and \cite{Bin}.\\


\end{abstract}

\section{Introduction and main results}

Stochastic Partial Differential Equations (SPDEs) have been used recently in many disciplines ranging from applied mathematics, statistical mechanics and theoretical physics to theoretical neuroscience, theory of complex chemical reactions (including polymer science), fluid dynamics and mathematical {finance} to quote only a few; see for example {\color{blue}\cite{Davar-CBMS}} and references therein.

In \cite {FoonNual}, the authors considered the following {stochastic heat equation},
\begin{equation}\label{eq1}
\partial_t u_t(x)=
 \mathcal{L}u_t(x)+ \xi\sigma (u_t(x)) \dot F(t,x),
\end{equation}
where $\mathcal{L}$ is the Dirichlet Laplacian on $B_R(0)$, the ball of radius $R$ {centered at the origin}. Under some appropriate conditions, it was shown that the long time behaviour of the solution is dependent on the {\it{noise level}}, that is on the values of $\xi$. More precisely, it was shown that for large values of $\xi$, the moments of the solution grow exponentially with time while for small values of $\xi$, the moments decay exponentially. In this { paper}, we extend the results of \cite{FoonNual} by taking $\mathcal{L}$ to be a non-local operator, the generator of a killed stable process, namely ${\mathcal{L}:= -\nu(-\Delta)^{\alpha/2}}\ \text{for} \ 0<\alpha\leq 2$ with {zero} exterior boundary conditions. We also provide some clarification and simplification of the proofs in \cite{FoonNual}. Non-local operators are becoming increasingly important due to their wide applicability for modeling purposes. The class of equations we study can for instance be used to model particles moving in a discontinuous fashion while being subject to { some} branching mechanism; see for example Walsh \cite{walsh}.

Throughout this paper, the initial condition $u_0$ is always {assumed} to be a non-negative bounded deterministic function such that for some set $K\subset B_R(0)$, the quantity
\begin{equation*}
\int_K u_0(x)\,\d x
\end{equation*}
is strictly positive.  The function $\sigma$ will be subjected to the following condition.
\begin{assumption}\label{cond}
The function $\sigma$ is assumed to be a globally Lipschitz function satisfying
$$
 l_\sigma |x|\leq {|\sigma(x)|}\leq L_\sigma |x| \quad\text{for all}\quad x\in \R,
 $$
for some positive constants $l_\sigma$ and $L_\sigma$.
%{\color{red} Should we instead assume the following with the absolute value of $\sigma$? $$
 %l_\sigma |x|\leq {|\sigma(x)|}\leq L_\sigma |x| \quad\text{for all}\quad x\in \R,
% $$}
\end{assumption}
Following Walsh \cite{walsh}, we look at the mild solution of \eqref{eq1} satisfying the following integral equation,
\begin{equation}\label{mild}
u_t(x)= (\mathcal{G}_{D} u_0)_t (x)+  \xi \int_{B_R(0)}\int_{0}^{t} p_D (t-s,x,y)\sigma(u_s(y))F(\d s{,}  \d y),
\end{equation}
where
\begin{align*}
(\mathcal{G}_{D} u_0)_t (x)=\int_{B_R(0)} u_0 (y)p_D (t,x,y)  \d y,
\end{align*}
and
\(p_D(t,x,y)\) denotes the heat kernel of the stable process.  When the driving noise is white in space and time, existence-uniqueness considerations impose the conditions that $d=1$ and $1<\alpha<2$. When  the noise term is not space-time white noise, it will be spatially correlated that is,
\begin{equation*}
\E\dot F(s,\,x)\dot F(t,\,y)=\delta_0(t-s)f(x,y),
\end{equation*}
where the correlation function $f$ satisfies the inequality $f(x,y)\leq \tilde{f}(x-y)$, and ${\tilde{f}}$ is a locally integrable positive continuous function on $\R^d\backslash \{0\}$ satisfying the following Dalang  type condition,
\begin{equation*}
\int_{\R^d}\frac{\hat{\tilde f}(\xi)}{1+|\xi|^\alpha}\,\d \xi<\infty,
\end{equation*}
where $\hat{\tilde f}$ denotes the Fourier transform of $\tilde{f}$; {see \cite{Dalang99}}. We will impose the following non-degeneracy condition on $f$,
\begin{assumption}\label{cond-f}
There exists a constant $K_R$ such that
\begin{equation*}
\inf_{x,y\in B_R(0)}f(x,y)\geq K_R.
\end{equation*}
\end{assumption}
The above conditions on the correlation function are  quite mild. Examples of correlation function{\color{red}s} satisfying {Assumption \ref{cond-f}} include the Riesz kernel, Cauchy kernels and many more{: See, for example,  \cite{FK2} and \cite{FLN}.} Our first set of results concerns equation \eqref{eq1} when the driving noise is space-time white noise which we denote by $\dot W$.  In other words, we are looking at
\begin{equation}\label{fspde}
\begin{cases}
\partial_t u_t(x)=\mathcal{L} u_t(x)+ \xi\sigma (u_t(x)) \dot W(t,x), \ \ \ \ x\in B_R(0), \ \ \ t>0\\
 u_t(x)=0, \ \  \  x\in B_R(0)^c.
\end{cases}
\end{equation}

\begin{theorem}\label{grw-dec1}
Let $u_t(x)$ be the unique solution of equation \eqref{fspde}, then there exists $\xi_0 >0 $ such that for all $\xi < \xi_0$ and $x \in B_R(0)$,
\\
$$-\infty < \limsup\limits_{t\to\infty }  \frac{1}{t} \log \E|u_t (x)|^2 < 0. $$
\\
Fix $\varepsilon >0$, then there exists $\xi_1>0$ such that for all $\xi > \xi_1$ and $x\in B_{R-\epsilon}(0)$,
$$0< \liminf\limits_{t\to\infty} \frac{1}{t}\log \E|u_t(x)|^2 < \infty. $$
\end{theorem}

As in \cite{FoonNual}, we define the energy of the solution by the following { quantity,}
\begin{equation}\label{energy}
\mathcal{E}_t (\xi)= \sqrt{\E{\|u_t\|}_{L^2(B_R(0))}^2}.
\end{equation}
The next corollary now follows easily from the above theorem.
\begin{corollary}\label{energ1}
With $\xi_0 $ and $\xi_1$ as in Theorem \ref{grw-dec1}, we have
\begin{align*}
-\infty<\limsup\limits_{t\to \infty} \frac{1}{t}\log \mathcal{E}_t (\xi)<0 \ \ \ \text{for all} \ \ \ \xi< \xi_0
\end{align*}
and
\begin{align*}
0<\liminf\limits_{t\to\infty}\frac{1}{t}\log\mathcal{E}_t (\xi)<\infty \ \ \ \text{for all }\ \ \ \xi>\xi_1.
\end{align*}
\end{corollary}
Our next set of results concern{s} equation \eqref{eq1} with colored noise satisfying the conditions above. That is, we consider
 \begin{equation}\label{frlapd}
 \begin{cases}
\partial_t u_t (x)= \mathcal{L}u_t(x)+ \xi \sigma(u_t (x)) \dot{F}(t,x),  \ \ \  x\in B_R (0),\ \ t>0
\\
u_t(x)=0 ,\ \ \ x\in {B_R(0)^c}.
\end{cases}
 \end{equation}
\\
  \begin{theorem}\label{gw-de-d}
    Assume that $u_t$ is the unique solution to equation \eqref{frlapd}. Then there exists $\xi_2>0$ such that for all $\xi<\xi_2$ and $x\in B_R(0)$
    \begin{align*}
    -\infty<\limsup\limits_{t\to\infty}\frac{1}{t}\log \E|u_t(x)|^2<0.
    \end{align*}
Fix $\varepsilon>0$, then there exists $\xi_3>0$ such that for all $\xi>\xi_3$ and $x\in B_{R-\varepsilon}(0)$,
    \begin{align*}
    0<\liminf_{t\to\infty}\frac{1}{t}\log \E|u_t (x)|^2 <\infty.
    \end{align*}
    \end{theorem}
We then have the following easy consequence.
   \begin{corollary}\label{energd}
   Let $\xi_2$ and $\xi_3$ be as in Theorem \ref{gw-de-d}, then
   \begin{align*}
   -\infty<\limsup\limits_{t\to\infty}\frac{1}{t}\log\mathcal{E}_t(\xi)<0 \ \ \ \text{for all}\ \ \xi<{\xi_2}
   \end{align*}
   and
   \begin{align*}
   0<\liminf\limits_{t\to\infty}\frac{1}{t}\log\mathcal{E}_t(\xi)<\infty\ \ \ \text{for all}\ \ \xi>{\xi_3}.
   \end{align*}
   \end{corollary}
We end this introduction with a plan of the article. In section 2, we provide some estimates needed for the proof{s} of our main results which are presented in section 3. Finally section 4 contains some extensions of our main results to higher  moments and to  some other non-local operators instead  of the  fractional Laplacian. { Throughout this paper, the letter c with or without subscript(s) will denote a constant whose value is not important and can vary from place to place.}
\section{Some estimates}
We begin this section with some estimates on heat kernel of the Dirichlet fractional Laplacian. For more information on these, see \cite{Bogdan} and references therein.
\begin{itemize}
\item \begin{equation}\label{pd0}
p_D (t,x,y)\leq c_1(t^{-\frac{d}{\alpha}}\wedge \frac{t}{|x-y|^{d+\alpha}}).
\end{equation}
We will often use the above inequality in the form of $p_D (t,x,y)\leq c_1p (t,x-y)$, where $p_t(\cdot)$ is the heat kernel of the (unkilled) stable process.
 \item Fix $\epsilon >0$ and let $x,\,y\in B_{R-\epsilon}(0)$, then for all $t\leq \epsilon^\alpha$,
 \begin{equation}\label{pd1}
 p_D (t,x,y)\geq c_2(t^{-\frac{d}{\alpha}}\wedge \frac{t}{|x-y|^{d+\alpha}}).
\end{equation}
\item There exist $t_0>0$ {and $\mu_1>0$} such that,
 \begin{equation}\label{pd2}
c_1 e^{-\mu_1 t} \leq
p_D (t,x,y) \leq c_2 e^{-\mu_1 t}{\quad \text{for}\quad t\geq t_0}.
\end{equation}
The upper bound is valid for any $x,\,y\in B_R(0)$ while the lower bound is valid for $x,\,y\in B_{R-\epsilon}(0)$ with $\epsilon>0$.
\end{itemize}
Our first lemma will be important for the white noise driven equation. The spatial dimension is restricted to $d=1$.
\begin{lemma}\label{lm2} There exists a constant $K_{\beta, \mu_1,\alpha}$ depending only on  $\beta$, $\mu_1$ and $\alpha$ such that for all $\beta \in (0,\mu_1)$ and $x\in B_R(0)$, we have
\begin{align*}
\int_{0}^{\infty} e^{\beta t} p_D (t,x,x)\d t \leq K_{\beta, \mu_1,\alpha}.
\end{align*}
\end{lemma}
\begin{proof}
We begin by writing
\begin{align*}
\int_{0}^{\infty}e^{\beta t}&p_D(t,x,x) \d t= \int_{0}^{t_0}e^{\beta t}p_D(t,x,x) \d t + \int_{t_0}^{\infty}e^{\beta t}p_D(t,x,x) \d t,
\end{align*}
where $t_0$ is as in \eqref{pd2}. Now using \eqref{pd0}, we have
\begin{align*}
\int_{0}^{t_0}e^{\beta t}p_D(t,x,x) \d t&\leq c_3\int_{0}^{t_0} e^{\beta t} t^{-\frac{1}{\alpha}} \d t,
\end{align*}
{where we have used} the fact that $d=1$. It is now clear that the { above} integral {has an upper bound depending on $\beta$} . Since $\beta <\mu_1$, we can use \eqref{pd2} to write
\begin{align*}
\int_{t_0}^{\infty}e^{\beta t}&p_D(t,x,x) \d t\\
&\leq c_5\int_{to}^{\infty} e^{-(\mu_1-\beta)t} \d t\\
&\leq \frac{c_6}{\mu_1-\beta}.
\end{align*}
Combining the estimates, we obtain the result.
\end{proof}

 \begin{lemma}\label{lmd1}
Let  $\beta\in(0,\mu_1)$ and $x\in B_R(0)$.{ Then there exists} a constant $c_{R,\,\alpha}$ depending on $R$ and $\alpha$ such that for all $t>0$
 \begin{align*}
 \int_{B_R(0)} e^{\beta t}p_D (t,x,y) \d y \leq c_{R,\,\alpha}.
 \end{align*}
  \end{lemma}
  \begin{proof}
Fix $t_0$ as in \eqref{pd2}.  For $0<t<t_0$, we have
    \begin{align*}
    \int_{B_R(0)} e^{\beta t} &p_D(t,x,y)\d y\\
   &\leq e^{\beta t}\int_{\mathbb{R}^d}p(t,x,y)\d y\\
    &\leq e^{\beta t_0},
    \end{align*}
    and for $t>t_0$ we use \eqref{pd2} to get
    \begin{align*}
    \int_{B_R(0)} e^{\beta t}& p_D(t,x,y)\d y\\
    &\leq c_2 e^{-(\mu_1-\beta)t_0}.
    \end{align*}
 The result  now easily follows from the  two inequalities { above}.
  \end{proof}

  \begin{lemma}\label{lem-cor}
  Let $\beta \in (0,\,2\mu_1)$. Then there exists a constant $c_{\beta,\mu_1}$ depending on $\beta$ and $\mu_1$ such that
  \begin{equation}
  \int_0^\infty e^{\beta t}\int_{B_R(0)\times B_R(0)}p_D(t,x_1,y_1)p_D(t,x_2,y_2)f(y_1,y_1) \d t\,\d y_1\,\d y_2\leq c_{\beta, \mu_1},
  \end{equation}
  for all $x_1,\, x_2,\  y_1, \ y_2 \in B_R(0)$.
  \end{lemma}
  \begin{proof}
  {We again  use \eqref{pd2} so we fix $t_0$ accordingly}.  We begin by splitting the integral as follows,
  \begin{align*}
  \int_0^\infty e^{\beta t}&\int_{B_R(0)\times B_R(0)}p_D(t,x_1,y_1)p_D(t,x_2,y_2)f(y_1,y_1) \d t\,\d y_1\,\d y_2 \\
  &=\int_0^{t_0} e^{\beta t}\int_{B_R(0)\times B_R(0)}p_D(t,x_1,y_1)p_D(t,x_2,y_2) f(y_1,y_1) \d t\,\d y_1\,\d y_2\\
  &+\int_{t_0}^\infty e^{\beta t}\int_{B_R(0)\times B_R(0)}p_D(t,x_1,y_1)p_D(t,x_2,y_2) f(y_1,y_1) \d t\,\d y_1\,\d y_2\\
  &:=I_1+I_2.
  \end{align*}
 $I_1$ can be bounded as follows{:  we  use }\eqref{pd0} to obtain
  \begin{align*}
  I_1&\leq e^{\beta t_0}\int_0^{t_0}e^{-\beta t}e^{\beta t}\int_{B_R(0)\times B_R(0)}p_D(t,x_1,y_1)p_D(t,x_2,y_2) f(y_1,y_1) \d t\,\d y_1\,\d y_2\\
 &\leq e^{2\beta t_0}\int_0^\infty e^{-\beta t}\int_{{\mathbb{R}^d\times \mathbb{R}^d}}p(t,x_1,y_1)p(t,x_2,y_2) \tilde{f}(y_1-y_1) \d t\,\d y_1\,\d y_2\\
 &\leq c_1e^{2\beta t_0}.
  \end{align*}
  The last inequality needs some justifications which are quite straightforward under the current conditions; see \cite{FK2} for details. For $I_2$, we use \eqref{pd2} to write
  \begin{align*}
  I_2\leq &\int_{t_0}^\infty e^{\beta t}\sup_{y_1,y_2\in B_R(0)}p_D(t,x_1,y_1)p_D(t,x_2,y_2){\ \d t}\int_{B_R(0)\times B_R(0)} f(y_1-y_1)\,\d y_1\,\d y_2\\
  &\leq c_2\int_0^\infty e^{-(2\mu_1-\beta)t}\,\d t.
  \end{align*}
  Combining the above estimates yield{s} the result.
  \end{proof}
 \begin{lemma}\label{lmd2}
Fix $\varepsilon>0$. Then, there exist $t_0>0$ and a constant $c_{\beta, \mu_1, t_0}$ such that for all $\beta>0$,
 \begin{align*}
 \int_{0}^{\infty} e^{-\beta t}p_D(t,x_1,y_1)p_D(t,x_2,y_2) \d t\geq c_{\beta, \mu_1, t_0},
 \end{align*}
 whenever $x_1,\, x_2,\  y_1, \ y_2 \in B_{R-\varepsilon}(0)$.  The constant $c_{\beta, \mu_1, t_0}$ depends on $\beta$, $\mu_1$ and $t_0$.
 \end{lemma}
 \begin{proof}
 Using \eqref{pd2}, we have
  \begin{align*}
  \int_{0}^{\infty} e^{-\beta t}&p_D(t,x_1,y_1)p_D(t,x_2,y_2) \d t\\
  &\geq \int_{t_0}^{\infty} e^{-\beta t}p_D(t,x_1,y_1)p_D(t,x_2,y_2) \d t \\
  &\geq  c_1 \int_{t_0}^{\infty} e^{-\beta t} e^{-2\mu_1 t} \d t\\
  &= \frac{c_2 e^{-(\beta +2\mu_1)t_0}}{\beta+2\mu_1}.
  \end{align*}
 \end{proof}

\section{Proofs of main results.}

 \subsection{Proof of Theorem \ref{grw-dec1}}
 \begin{proof}[Proof of Theorem \ref{grw-dec1}.]
 Using \eqref{mild}{} {and the Walsh isometry,} we have
\begin{equation}\label{mild-sq}
\E|u_t (x)|^2= |(\mathcal{G}_{D}u_{0})_t (x)|^2 +\xi^2 \int_{0}^{t}{\int_{B_R(0)}} {p_D}^2(t-s,x,y)\E|\sigma(u_s (y))|^2 \d y\d s.
\end{equation}
from which we obtain
\begin{equation}\label{jensen's^2}
\E|u_t (x)|^2 \geq |(\mathcal{G}_{D}u_{0})_t (x)|^2.
\end{equation}
Using the assumption on $u_0$, we have for $\epsilon>0$ small enough,
\begin{align*}
(\mathcal{G}_{D}u_{0})_t (x)&=\int_{B_R(0)}u_0(y)p_D(t,x,y)\,\d y\\
&\geq \int_{B_{R-\epsilon}(0)}u_0(y)p_D(t,x,y)\,\d y\\
&\geq c_1e^{-\mu_1t},
\end{align*}
whenever $t\geq t_0$ with $t_0$ as in \eqref{pd2} { and $x\in B_{R-\epsilon}(0)$}.
This immediately gives
\begin{align*}
 \liminf\limits_{t\to\infty }  \frac{1}{t} \log \E|u_t (x)|^2 > -\infty\quad { for\quad x\in B_{R-\epsilon}(0).}
\end{align*}
We now look at the upper bound. We will assume that $\beta\in (0,\,2\mu_1)$. From \eqref{mild-sq} and the assumption on $\sigma$,  we have
\begin{align*}
\E|u_t (x)|^2&\leq |(\mathcal{G}_{D}u_{0})_t (x)|^2 +\xi^2L_\sigma^2 \int_{0}^{t}\int_{B_R(0)} {p_D}^2(t-s,x,y)\E|u_s (y)|^2 \d y\d s\\
&:=I_1+I_2.
 \end{align*}
Using Lemma \ref{lmd1}, we have
\begin{align*}
I_1&\leq  c_2 e^{-\beta t} \left|\int_{B_R(0)} e^{\frac{\beta t}{2}} p_D (t,x,y)  \d y\right|^2 \\
&\leq  c_2 e^{-\beta t}.
\end{align*}
We { then} look at the second term $I_2$. {Using the semigroup property and Lemma \ref{lm2}, we have}
\begin{align*}
 I_2 &=  \xi^2 L_{\sigma}^2 e^{-\beta t} \int_{0}^{t}\int_{B_R(0)} e^{\beta (t-s)} p_D^2 (t-s, x,y)  e^{\beta s} \E|u_s (y)|^2 \d y\d s\\
 &\leq \xi^2 L_{\sigma}^2 e^{-\beta t} \sup_{t>0,\,x\in B_R(0)}e^{\beta t}\E|u_t (x)|^2 \int_{0}^{t}\int_{B_R(0)} e^{\beta (t-s)} p_D^2 (t-s, x,y)  \d y\d s\\
 &\leq \xi^2 L_{\sigma}^2 e^{-\beta t} \sup_{t>0,\,x\in B_R(0)}e^{\beta t}\E|u_t (x)|^2 \int_{0}^{t}{ e^{\beta s} p_D (2s, x,x) \d s}\\
&\leq K_{\beta, \mu_1,\alpha}\xi^2 L_{\sigma}^2 e^{-\beta t} \sup_{t>0,\,x\in B_R(0)}e^{\beta t}\E|u_t (x)|^2.
\end{align*}
Combining the above inequalities, we have
\begin{align*}
\sup_{t>0,\,x\in B_R(0)}e^{\beta t}\E|u_t (x)|^2\leq c_2+K_{\beta, \mu_1,\alpha}\xi^2 L_{\sigma}^2 \sup_{t>0,\,x\in B_R(0)}e^{\beta t}\E|u_t (x)|^2.
\end{align*}
We now choose $\xi_0$ such that for $\xi\leq \xi_0$, we have $K_{\beta, \mu_1,\alpha}\xi^2 L_{\sigma}^2<\frac{1}{2}$.
This immediately gives
\begin{align*}
\limsup\limits_{t \to \infty} \frac{1}{t} \log \E|u_t(x)|^2<0.
\end{align*}
We have thus proved the first half of the theorem. For the second { half}, {we look at the} following `Laplace transform',
\begin{equation*}
I_\beta:=\int_0^\infty e^{-\beta t}\inf_{x\in B_{R-\epsilon}(0)}\E|u_t(x)|^2\,\d x.
\end{equation*}
Using the mild formulation and the condition on $\sigma$, we have
\begin{align*}
\E|u_t (x)|^2 \geq & |(\mathcal{G}_{D}u_{0})_t (x)|^2 +\xi^2 l_{\sigma}^2 \int_{0}^{t}\int_{B_R(0)} {p_D}^2(t-s,x,y)\E|u_s (y)|^2 \d y \d s.
\end{align*}
From the above, we have $ I_\beta \geq I_1+ I_2$, where $I_1$ and $I_2$ are Laplace transforms of the first and second term of the above display respectively. We look at $I_1$ first. Note that for fixed $\epsilon >0$,
\begin{align*}
\inf_{x\in B_{R-\epsilon}(0)}(\mathcal{G}_{D}u_{0})_t (x)&\geq \int_{B_{R-\epsilon}(0)}u_0(y)p_D(t,x,y)\,\d y\\
&\geq c_4 \inf_{x,y\in B_{R-\epsilon}(0)}p_D(t,x,y).
\end{align*}
Using \eqref{pd2}, for $t\geq t_0$, we have
\begin{align*}
I_1&\geq \int_{t_0}^\infty e^{-\beta t}{\inf\limits_{x\in B_{R-\epsilon}(0)}}|(\mathcal{G}_{D}u_{\color{red}0})_t (x)|^2\,\d t\\
&\geq \frac{c_3e^{-(\beta+2\mu_1)t_0}}{\beta+2\mu_1}.
\end{align*}
For the second term, we obtain
\begin{align*}
I_2&\geq\xi^2 l_{\sigma}^2 I_\beta \int_{t_0}^\infty e^{-\beta s} \inf_{x\in B_{R-\epsilon}(0)}p_D^2(s,x,y)\,\d y\\
&\geq c_5\xi^2 l_{\sigma}^2 I_\beta\frac{ e^{-(\beta+2\mu_1)t_0}}{\beta+2\mu_2}.
\end{align*}
Combining the above inequalities yields
\begin{equation*}
I_\beta\geq \frac{c_3e^{-(\beta+2\mu_1)t_0}}{\beta+2\mu_1}+c_5\xi^2 l_{\sigma}^2 I_\beta\frac{ e^{-(\beta+2\mu_1)t_0}}{\beta+2\mu_2}.
\end{equation*}
We can {now} choose $\xi_1$ large enough so that for $\xi\geq {\xi_1}$, we have
\begin{equation*}
I_\beta\geq \frac{c_3e^{-(\beta+2\mu_1)t_0}}{\beta+2\mu_1}+2 I_\beta,
\end{equation*}
which gives us $I_\beta=\infty$. This proves
\begin{align*}
\liminf\limits_{t\to \infty} \frac{1}{t} \log \E|u_t (x)|^2 > 0.
\end{align*}
The fact that
\begin{align*}
\liminf\limits_{t\to \infty} \frac{1}{t} \log \E|u_t (x)|^2 <\infty
\end{align*}
easily follows from the ideas in \cite{FK}. We leave it to the reader to fill in the details.
 \end{proof}
 \subsection{Proof of Corollary \ref{energ1}}
 \begin{proof}[Proof of Corollary \ref{energ1}.]
The proof follows essentially from Theorem \ref{grw-dec1} and the definition of the energy of the solution together with the following estimate
\begin{align*}
|B_{R-\epsilon}(0)|\inf\limits_{x\in B_{R-\epsilon}(0)}{\E}|u_t (x)|^2 \leq \int_{B_R(0)} {\E}|u_t (x)|^2 dx \leq |B_R(0)| \sup\limits_{x\in B_R(0)} {\E}|u_t (x)|^2.
\end{align*}
 \end{proof}

 \subsection{Proof of Theorem \ref{gw-de-d}}
While { one can expect the proof of Theorem \ref{gw-de-d} to follow a similar pattern to that of Theorem \ref{grw-dec1}, the noise term is now colored thus the proof is harder and requires a new idea}. We provide the details of the proof of Theorem \ref{gw-de-d}. The proof of Corollary \ref{energd} is omitted since it is similar to that of  Corollary \ref{energ1}.
\begin{proof}[Proof of Theorem \ref{gw-de-d}.]
Using the mild formulation of the solution and the assumption on $\sigma$, we obtain
  \begin{align*}\label{moment-d}
  \E|u_t(x)|^2 &=|(\mathcal{G}_{D} u_{0})_t (x)|^2+  \xi^2 \int_{0}^{t}\int_{B_R(0)\times B_R(0)} p_D (t-s,x,y_1)p_D(t-s,x,y_2){f(y_1,y_2)}
   \\
  &\times \E| \sigma(u_s(y_1))\sigma(u_s(y_2))|\d y_1 \d y_2 \d s\\
 &\leq  |(\mathcal{G}_{D} u_{0})_t (x)|^2+  \xi^2 L_\sigma^2 \int_{0}^{t}\int_{B_R(0)\times B_R(0)} p_D (t-s,x,y_1)p_D(t-s,x,y_2){f(y_1,y_2)}
    \\
  &\times \E|u_s(y_1)u_s(y_2)|\d y_1 \d y_2 \d s
  \\
  &=  I_1+ I_2.
  \end{align*}
Set $\beta\in (0,\,2 \mu_1)$. Take $t_0$ as in \eqref{pd2}.  As is the proof of  Theorem \ref{grw-dec1}, we have
\begin{equation*}
I_1\leq c_1e^{-\beta t}\quad \text{whenever}\quad t>t_0.
\end{equation*}
We now bound $I_2$ by using Lemma \ref{lem-cor}.
\begin{align*}
I_2&\leq \xi^2 L_\sigma^2e^{-\beta t}\sup_{t>0, x\in B_R(0)}e^{\beta t}\E|u_t(x)|^2\\
&\times \int_0^\infty e^{\beta t}\int_{B_R(0)\times B_R(0)} p_D (t,x,y_1)p_D(t,x,y_2){f(y_1,y_2)}\d y_1 \d y_2 \d t\\
&\leq c_2\xi^2 L_\sigma^2e^{-\beta t}\sup_{t>0, x\in B_R(0)}e^{\beta t}\E|u_t(x)|^2.
\end{align*}
Using the two bounds above, we can use the arguments of the first part of the proof of Theorem \ref{grw-dec1} to show that
\begin{align*}
\limsup\limits_{t \to \infty} \frac{1}{t} \log \E|u_t(x)|^2<0 .
\end{align*}
The first part of our first theorem also gives us
\begin{align*}
 \liminf\limits_{t\to\infty }  \frac{1}{t} \log \E|u_t (x)|^2 > -\infty\quad { \text{for}\quad x\in B_{R-\epsilon}(0).}
\end{align*}
We now turn our attention to the final part of the proof. Fix $\beta,\epsilon>0$ {and} consider the following 'Laplace transform',
  \begin{align*}
  J_{\beta}:= \int_{0}^{\infty} e^{-\beta t}\inf\limits_{x,y\in B_{R-\varepsilon}(0)} \E|u_t(x)u_t(y)| \d t.
  \end{align*}
From the mild solution, we have
\begin{align}\label{cov}
  \E\big(u_t(x_1)u_t(x_2)\big)= (\mathcal{G}_D u_{0})_t (x_1)(\mathcal{G}_D u_{0})(x_2)
  + \xi^2\int_{0}^{t}{\int_{B_R(0)\times B_R(0)} }p_D(t-s,x_1,y_1)\nonumber\\
  \times p_D(t-s,x_2,y_2){f(y_1,y_2)}\E(\sigma(u_s(y_1))\sigma(u_s(y_2)))\d y_1\d y_2 \d s.
  \end{align}
Using the condition on $\sigma$, we have
\begin{align*}
 \E\big(&|u_t(x_1)u_t(x_2)|\big)\geq | (\mathcal{G}_D u_{0})_t (x_1)(\mathcal{G}_D u_{0})(x_2) |\\
 &+ \xi^2l_\sigma^2\int_{0}^{t}\int_{B_R(0)\times B_R(0)} p_D(t-s,x_1,y_1)p_D(t-s,x_2,y_2){f(y_1,y_2)}\E|u_s(y_1)u_s(y_2)|\d y_1\d y_2 \d s\\
 &:=J_1+J_2.
\end{align*}
We bound $J_2$ first by using the condition on the correlation function.
\begin{align*}
J_2&\geq \xi^2l_\sigma^2K_R\int_{0}^{t}\int_{B_{{R-\epsilon}}(0)\times B_{{R-\epsilon}}(0)} p_D(t-s,x_1,y_1)p_D(t-s,x_2,y_2)\E(|u_s(y_1)u_s(y_2)|)\d y_1\d y_2 \d s\\
&\geq c_3\xi^2l_\sigma^2K_R \int_0^t\inf_{y_1,\,y_2\in B_{{R-\epsilon}}(0)} p_D(t-s,x_1,y_1)p_D(t-s,x_2,y_2)\E(|u_s(y_1)u_s(y_2)|)\,\d s.
\end{align*}
Using these estimates, we have
\begin{align*}
  J_{\beta}\geq \tilde{J}_1+\tilde{J}_2,
\end{align*}
where $\tilde{J}_1$ and $\tilde{J}_2$ are the Laplace transform{s} of $J_1$ and $J_2$ respectively.
As in the proof of Theorem \ref{grw-dec1}, we have
\begin{align*}
\tilde{J}_1&\geq \int_0^\infty e^{-\beta t}| (\mathcal{G}_D u_{0})_t (x_1)(\mathcal{G}_D u_{0})(x_2) |\,\d t\\
&\geq \frac{c_4e^{-(\beta+2\mu_1)t_0}}{\beta+2\mu_1}{\quad\text{for}\quad x_1,\,x_2\in B_{R-\epsilon}(0)}.
\end{align*}
$\tilde{J}_2$ can be estimated using Lemma \ref{lmd2} as follows.
\begin{align*}
\tilde{J}_2\geq c_4\xi^2l_\sigma^2K_RJ_\beta \frac{e^{-(\beta+2\mu_1)t_0}}{\beta+2\mu_1}.
\end{align*}
We therefore have
\begin{align*}
J_\beta\geq\frac{c_4e^{-(\beta+2\mu_1)t_0}}{\beta+\mu_1}+c_4\xi^2l_\sigma^2K_RJ_\beta \frac{e^{-(\beta+2\mu_1)t_0}}{\beta+2\mu_1}.
\end{align*}
Therefore there exists a $\xi_3>0$ such that we have $J_\beta=\infty$ for $\xi\geq \xi_3$. Using the ideas above, we have
\begin{align*}
\int_0^\infty e^{-\beta t}\E |u_t(x)|^2\,\d t\geq c_5K_RJ_\beta \frac{e^{-(\beta+2\mu_1)t_0}}{\beta+2\mu_1} \quad { \text{for}\quad x\in B_{R-\epsilon}(0)}.
\end{align*}
Therefore for $\xi\geq \xi_3$, we obtain
\begin{align*}
\int_0^\infty e^{-\beta t}\E |u_t(x)|^2\,\d t=\infty,
\end{align*}
which implies that
\begin{align*}
\liminf\limits_{t\to \infty} \frac{1}{t} \log \E|u_t (x)|^2 > 0\quad { \text{for}\quad x\in B_{R-\epsilon}(0).}.
\end{align*}
Again the ideas of \cite{FK2} give
\begin{align*}
\liminf\limits_{t\to \infty} \frac{1}{t} \log \E|u_t (x)|^2 <\infty.
\end{align*}
The theorem is therefore proved.
\end{proof}
\section{Some extensions}
We conclude this paper with some extensions that can be proved using the  methods developed in our paper.
Since {o}ur { main theorems (Theorem \ref{grw-dec1} \& Theorem \ref{gw-de-d} ) are  about  second moments of the solution to the corresponding equation}, one may naturally  {ask}  if they also hold for higher moments. This is actually answered in the following theorems.
\begin{theorem}\label{grw-decp}

If  $u_t$ is the unique solution to \eqref{fspde}, then for all $p\geq 2$, there exists $\xi_0 (p){\color{red}>0}$ such that for all $\xi<\xi_0 (p)$ and $x\in B_R(0)$,
$$-\infty < \limsup\limits_{t\to\infty }  \frac{1}{t} \log {\E}|u_t (x)|^p < 0. $$
\\
On the other hand, for all $\varepsilon >0$, there exists $\xi_1 (p){\color{red} >0}$ such that for all $\xi > \xi_1(p)$ and $x\in B_{R-\epsilon}(0)$,
$$0< \liminf\limits_{t\to\infty}{\frac{1}{t}}\log {\E}|u_t(x)|^p < \infty.$$
\end{theorem}
The proof of this theorem follows { from} the  Burkh{\"{o}}lder-Davis-Gundy inequality for the upper bound and Jensen's inequality for the lower bound. {We do not provide a proof here. The reader can refer to \cite{FoonNual} for details.}  Next we state   a result similar to Theorem  \ref{grw-decp} in higher dimension.
\begin{theorem}\label{gr-de-d-dimp}
Let $u_t(x)$ be the unique solution to \eqref{frlapd}, then  for all $p\geq 2$ there exists $\xi_2(p)>0$ such that for all $\xi<\xi_2(p)$ and $x\in B_R(0)$
\begin{align*}
-\infty<\limsup\limits_{t\to\infty}\frac{1}{t}\log {\E}|u_t(x)|^p<0.
\end{align*}
On the other hand, for all $\varepsilon>0$, there exists $\xi_3(p)>0$ such that for all $\xi>\xi_3(p)$ and $x\in B_{R-\varepsilon}(0)$
\begin{align*}
0<\liminf\limits_{t\to\infty}\frac{1}{t}\log {\E}|u_t(x)|^p<\infty.
\end{align*}
\end{theorem}
In the remainder of this section, we show that the method developed in this paper can be used to study   problems with {operators other than the fractional Laplacian.}
 \begin{exmp}
 Consider the following stochastic {heat} equation with linear drift
 \begin{equation}\label{fspde1}
 \begin{cases}
 \partial_t u_t(x)=
  -{\nu}(-\Delta)^{\frac{\alpha}{2}} u_t(x)+\lambda u_t(x)+ \xi\sigma (u_t(x)) \dot W(t,x), \ \ \ \ x\in B_R(0), \ \ \ t>0\\
  u_t(x)= 0, \ \  \ x\in B_R(0)^{c}, \ \  t>0
 \end{cases}
 \end{equation}
 where $\lambda$ is a real number and all the other  conditions are the same as in \eqref{fspde}. { The mild solution is given by}
 \begin{equation}
  u_t(x)=(\mathcal{G}_{D}^* u_{0})_t (x)+  \xi \int_{B_R(0)}\int_{0}^{t} p_D^* (t-s,x,y)
    \sigma(u_s(y))F({\d}s,{\d}y)
 \end{equation}
 where $$p_D^*(t,x,y)= e^{\lambda t} p_D(t,x,y)
 $$ and
 $$(\mathcal{G}_D^* u_{ 0})_t(x)= \int_{B_R(0)}u_0(y)p_D^*(t,x,y) {\d}y.
 $$
 Then there exists $\xi_0(\lambda)>0$ such that for all $\xi<\xi_0(\lambda)$ and $x\in B_R(0)$
 $$-\infty<\limsup\limits_{t\to\infty}\frac{1}{t}\log {\E}|u_t(x)|^2<0
 $$
 while for $\varepsilon>0$ there exists $\xi_1(\lambda)>0$ such that for all $\xi>\xi_1(\lambda)$ and $x\in B_{R-\epsilon}(0)$
 $$0<\liminf\limits_{t\to\infty}\frac{1}{t}\log {\E}|u_t(x)|^2<\infty.
 $$
 The proof is very similar to that of  Theorem \ref{grw-dec1}, we only need to adjust  Lemma \ref{lm2} %and Lemma \ref{lm2}
  as follows:
 $$\int_{0}^{\infty} e^{\beta t}p_D^*(t,x,x){\d}t\leq c_1\Bigg[(\lambda+\beta)^{-\frac{1}{\alpha}}+ \frac{1}{\mu_1-(\lambda +\beta)}\Bigg].
 $$
 %and
% $$\sup\limits_{t>0}\int_{0}^{L} e^{\beta t} p_D^*(t,x,y)dy \leq c_2
% $$
 for all $\beta >0$, $x\in B_R(0)$ provided $0<\lambda +\beta <\mu_1$.  Now if $x,y\in B_{R-\epsilon}(0)$ and $2(\mu_1-\lambda)+\beta>0$, we have
 $$\int_{0}^{\infty} e^{-\beta t}(p_D^*(t,x,y))^2 {\d}t\geq c_2 \frac{e^{-(2\mu_1-2\lambda+\beta)t_0}}{2\mu_1+2\lambda-\beta},
 $$
 where $c_1$ and $c_2$ are some positive constants.  With some modifications of the proofs above, Theorems  \ref{grw-dec1}, \ref{gw-de-d}, \ref{grw-decp} and  \ref{gr-de-d-dimp}  hold for the solution of  equation \eqref{fspde1}.
 \end{exmp}

 \begin{exmp}
 In this example we consider the generator of a relativistic stable process killed upon exiting $B_R(0)$ instead of the fractional Laplacian. We therefore look at
 \begin{equation}\label{rela}
 \begin{cases}
 \partial_t u_t(x)= m u_t(x)- (m^{\frac{2}{\alpha}}-\Delta)^{\frac{\alpha}{2}} u_t(x)+\xi\sigma(u_t(x)) \dot F(t,x), \ \ \ x\in B_R(0), \ \ t>0\\
 u_t(x)= 0, \ \ \ x\in {B_R(0)^c}.
 \end{cases}
 \end{equation}
{ {H}ere $m$ is some fixed positive number and all the other conditions are the same as in \eqref{gw-de-d}. { We refer  the reader to \cite{relativistic} for the needed heat kernel bounds to prove { appropriate versions of} Theorems  \ref{grw-dec1}, \ref{gw-de-d}, \ref{grw-decp} and  \ref{gr-de-d-dimp}.} We leave it for the reader to fill in the details.}
 \end{exmp}

 \begin{exmp}
 We conclude this section with this interesting problem. Let $1<\beta<\alpha<2$ and consider the following:
 \begin{equation}\label{double-delta}
 \begin{cases}
 \partial u_t(x)= -{\nu}(-\Delta)^{\frac{\alpha}{2}} u_t(x)-a^\beta(-\Delta)^{\frac{\beta}{2}}u_t(x)+\xi\sigma(u_t(x))\dot F(t,x), \ \ \ x\in B_R(0), \ \ t>0\\
 u_t(x)= 0, \ \ \ x\in {B_R(0)^c}.
 \end{cases}
 \end{equation}
{Here we refer the reader to  \cite{double-laplacian} { for the heat kernel bounds needed to prove  suitable versions of } Theorems  \ref{grw-dec1}, \ref{gw-de-d}, \ref{grw-decp} and  \ref{gr-de-d-dimp} for  the solution of the above equation. We leave it for the reader to fill in the details.}
 \end{exmp}

\bibliography{Foon}
\end{document}